\documentclass[12pt]{article}
\usepackage[colorlinks]{hyperref}
\usepackage{color}
\usepackage{graphicx}
\usepackage{graphics}
\usepackage{makeidx}
\usepackage{showidx}
\usepackage{latexsym}
\usepackage{amssymb}
\usepackage{verbatim}
\usepackage{amsmath}
\usepackage{amsthm}
\usepackage{amsfonts}
\usepackage{amssymb,amsmath}
\usepackage{latexsym,amsthm,amscd}
\usepackage[all]{xy}

\newtheorem{prop}{Proposition}[section]
\newtheorem{defn}{Definition}[section]
\newcounter{alphthm}
\newtheorem{propriete}[alphthm]{Theorem}

\newtheorem{Corollary}[alphthm]{Corollary}

\newtheorem{thm}{Theorem}
\newtheorem{lem}{Lemma}[section]
\newcommand{\be}{\begin{equation}}
\newcommand{\ee}{\end{equation}}
\newcommand{\ben}{\begin{enumerate}}
\newcommand{\een}{\end{enumerate}}
\newcommand{\beq}{\begin{eqnarray}}
\newcommand{\eeq}{\end{eqnarray}}
\newcommand{\beqn}{\begin{eqnarray*}}
\newcommand{\eeqn}{\end{eqnarray*}}

\newcommand{\bpf}{\begin{proof}}
\newcommand{\epf}{\end{proof}}
\newcommand{\bl}{\begin{lem}}
\newcommand{\el}{\end{lem}}
\newcommand{\bp}{\begin{prop}}
\newcommand{\ep}{\end{prop}}
\newcommand{\bd}{\begin{defn}}
\newcommand{\ed}{\end{defn}}
\newcommand{\bt}{\begin{thm}}
\newcommand{\et}{\end{thm}}

\def\nn{\nonumber}

\newcommand\bpr{\begin{prop}}
\newcommand\epr{\end{prop}}

\title{Evolution of Ricci scalar under Finsler Ricci flow}
\author{B. Bidabad\thanks{The corresponding author, bidabad@aut.ac.ir}\, and\, M. K. Sedaghat}
\date{}
 
\begin{document}
\maketitle
\noindent

\begin{abstract}
Recently, we have studied evolution of a family of Finsler metrics along Finsler Ricci flow and proved its convergence in short time. Here,  evolution equation of the reduced $hh$-curvature and the Ricci scalar along the Finslerian Ricci flow is obtained and it is proved that the Ricci flow preserves positivity of reduced $hh$-curvature on finite time. Next, it is shown that the evolution of Ricci scalar is a parabolic-type equation and if the initial Finsler metric is of positive flag curvature, then the flag curvature and the Ricci scalar remain positive as long as the solution exists. Finally, a lower bound for the Ricci scalar along the Ricci flow is obtained.
 \end{abstract}
\vspace{.5cm}
\section*{Introduction}
In the last decades, geometric flows and more notably among them, Ricci flow, are proved to be useful tools in the study of long standing conjectures in geometry and topology of the base Riemannian manifold. One of its important issues concerns discovering the so called round metrics (of constant curvature, Einstein, Solitons, .. etc.) on manifolds by evolving an initial Riemannian metric tensor to make it rounder and draw geometric and topological conclusions from the final round metric. Similarly, several natural questions are revealed in Finsler geometry, among them S. S. Chern's question stating that weather there exists a Finsler-Einstein metric on every smooth manifold.

Assembling an evolution equation in Finsler geometry contains a number of new conceptual and fundamental issues on relation with the different definitions of Ricci tensors, the existence problem and then geometrical and physical characterization. In \cite{Bao}, D. Bao based on the Akbar-Zadeh's Ricci tensor and in analogy with the Ricci flow in Riemannian geometry, has considered the following equation as Ricci flow in Finsler geometry
\begin{eqnarray}\label{RF}
\frac{\partial}{\partial t}\log F=-\mathcal{R}ic,\qquad F(0)=F_{0},
\end{eqnarray}
where, $F_{0}$ is the initial Finsler structure. This equation addresses the evolution of the Finsler structure $F$ and seems to make sense, as an unnormalized Ricci flow for Finsler spaces on both the manifolds of nonzero tangent vectors $TM_{0}$ and the sphere bundle $SM$. One of the advantages of (\ref{RF}) is its independence to the choice of Cartan, Berwald or Chern connections.

Recently, we have studied Finsler Ricci solitons as a self similar solutions to the Finsler Ricci flow and it was shown if there is a Ricci soliton on a compact Finsler manifold then there exists a solution to the Finsler Ricci flow equation and vice-versa, see \cite{BY}. 
 Next, as a first step to answer  Chern's question, we have considered evolution of a family of Finsler metrics, first under a general flow next under Finsler Ricci flow and prove that a family of Finsler metrics $g(t)$ which are solutions to the Finsler Ricci flow converge to a smooth limit Finsler metric as $t$ approaches the finite time $T$, see \cite{YB2}.
Moreover, a Bonnet-Myers type theorem was studied and it is proved that on a Finsler space, a forward complete shrinking Ricci soliton is compact if and only if the corresponding vector field is bounded, using which we have shown a compact shrinking Finsler Ricci soliton has finite fundamental group and hence the first de Rham cohomology group vanishes, see \cite{YB1}. The  existence of solution to the  evolution equation (\ref{RF}) in Finsler geometry, is also studied by the present authors in \cite{BKS}.

In the present work, we derive evolution equations for the reduced $hh$-curvature of Finsler structure $R(X,Z)$ and  the Ricci scalar $\mathcal{R}ic$ along the Ricci flow and show that the evolution of Ricci scalar is a parabolic type equation.
Next we show that if $(M,F(0))$ has positive reduced $hh$-curvature at the initial time  $t=0$ then, $(M,F(t))$ has positive reduced $hh$-curvature for all $t\in[0,T)$ and among the others show the following theorems.

\begin{thm}\label{asli4}
Let $(M^{n},F_{0})$ be a compact Finslerian manifold and $F(t)$ a solution to the evolution equation (\ref{RF}), satisfying a uniform bound for the Ricci tensor on a finite time interval $[0,T)$, where $F(0)=F_{0}$. If $(M,F(0))$ is of positive flag curvature, then $(M,F(t))$ has positive flag curvature and positive Ricci scalar for all $t\in[0,T)$.
\end{thm}

\begin{thm}\label{asli5}
Let $(M^{n},F_{0})$ be a compact Finslerian manifold and $F(t)$ a solution to the evolution equation (\ref{RF}), satisfying a uniform bound for the Ricci tensor on a finite time interval $[0,T)$, where $F(0)=F_{0}$. If $(M,F(0))$ has positive flag curvature and $\inf_{SM}\mathcal{R}ic_{g(0)}=\alpha>0$, then $\mathcal{R}ic_{g(t)}\geq \frac{\alpha}{1+\alpha t}$ for all $t\in[0,T)$.
\end{thm}
\section{Preliminaries and notations}
In order to study evolution equations in Finsler geometry, in analogy with Riemannian geometry, it is more convenient to use global definitions of curvature tensors. In the present work, whenever we are dealing with Cartan connection, we use notations and terminologies of \cite{HAZ}, otherwise we use those of \cite{BCS}.
Here and everywhere in this paper all  manifolds are supposed to be closed (compact and without boundary).
\subsection{Cartan connection on Finsler spaces}
Let $M$ be a real n-dimensional manifold of class $C^{\infty}$. We denote by $TM$ the tangent bundle of tangent vectors,  by  $\pi :TM_{0}\longrightarrow M$ the fiber bundle of non-zero tangent vectors and  by $\pi^*TM\longrightarrow TM_0$ the pull back tangent bundle.
Let $F$ be a Finsler structure on $TM_{0}$ and $g$ the related Finslerian metric. A \emph{Finsler manifold} is denoted here by the pair $(M,F)$. Any point of $TM_0$ is denoted by $z=(x,y)$, where $x=\pi z\in M$ and $y\in T_{\pi z}M$. We denote by $TTM_0$, the tangent bundle of $TM_0$ and by $\varrho$, the canonical
linear mapping $\varrho:TTM_0\longrightarrow \pi^*TM,$ where, $ \varrho=\pi_*$. For all $z\in TM_0$, let $V_zTM$ be the set of all vertical vectors at $z$, that is, the set of vectors which are tangent to the fiber through $z$.

Let  $\nabla:{\cal X}(TM_0)\times\Gamma(\pi^{*}TM)\longrightarrow\Gamma(\pi^{*}TM)$, be a linear connection. Consider the linear mapping
$\mu:TTM_0\longrightarrow \pi^*TM,$ by $\mu(\hat{X})=\nabla_{\hat{X}}u$ where, $\hat{X}\in TTM_0$ and $u=y^{i}\frac{\partial}{\partial x^{i}}$ is the canonical section of $\pi^*TM$.
The connection $\nabla$ is said to be {\it regular}, if $\mu$ defines an isomorphism between $VTM_0$ and
$\pi^*TM$. In this case, there is the horizontal distribution $HTM$ such that we have the Whitney sum $TTM_0=HTM\oplus VTM.$ This decomposition permits to write a vector field $\hat{X}\in {\cal X}(TM_0)$ into the horizontal and vertical form $\hat{X}=H\hat{X}+V\hat{X}$ uniquely. In the sequel, we will denote all the vector fields on $TM_0$ by $\hat{X}, \hat{Y}$, etc and the corresponding sections of $\pi^*TM$ by $X=\varrho(\hat X)$, $Y=\varrho(\hat Y)$, etc respectively, unless otherwise specified. The structural equations of the regular connection $\nabla$ are given by:
\begin{eqnarray*}
&&\tau(\hat{X},\hat{Y})=\nabla_{\hat{X}}Y-\nabla_{\hat{Y}}X-\varrho[\hat{X},\hat{Y}],\\
&&\Omega(\hat{X},\hat{Y})Z=\nabla_{\hat{X}}\nabla_{\hat{Y}}Z-\nabla_{\hat{Y}}\nabla_{\hat{X}}Z
-\nabla_{[\hat{X},\hat{Y}]}Z,
\end{eqnarray*}
where, $X=\varrho(\hat{X})$, $Y=\varrho(\hat{Y})$, $Z=\varrho(\hat{Z})$ and $\hat{X}$, $\hat{Y}$ and $\hat{Y}$ are vector fields on $TM_0$.
The torsion tensor $\tau$ and the curvature tensor $\Omega$ determine the two torsion tensors denoted here by $S$ and $T$ and the three
curvature tensors denoted by $R$, $P$ and $Q$ defined by:
\begin{eqnarray*}
&S(X,Y)=\tau(H\hat{X},H\hat{Y}),&\ \ \ T(\dot{X},Y)=\tau(V\hat{X},H\hat{Y}),\\
&R(X,Y)=\Omega(H\hat{X},H\hat{Y}),&\ \ \ P(X,\dot{Y})=\Omega(H\hat{X},V\hat{Y}),\\
&Q(\dot{X},\dot{Y})=\Omega(V\hat{X},V\hat{Y}),&\nonumber
\end{eqnarray*}
where, $X=\varrho(\hat{X})$,\ $Y=\varrho(\hat{Y})$,\
$\dot{X}=\mu(\hat{X})$ and $\dot{Y}=\mu(\hat{Y})$. The tensors $R$, $P$ and $Q$ are called $hh-$, $hv-$ and $vv-$curvature tensors, respectively.
There is a unique regular connection called the {\it Cartan connection} satisfying the metric compatibility and the $hh$-torsion freeness conditions in the following senses, see \cite{HAZ}.
\begin{align*}
&\nabla_{\hat{Z}}g=0,\\
&S(X,Y)=0,\\
&g(\tau(V\hat{X},\hat{Y}),Z)=g(\tau(V\hat{X},\hat{Z}),Y).
\end{align*}
Given an induced natural coordinates on $\pi^{-1}(U)$, we denote by $G^{i}$ the components of spray vector field on $TM$, where $G^{i}=\frac{1}{4}g^{ih}(\frac{\partial^{2}F^{2}}{\partial y^{h}\partial x^{j}}y^{j}-\frac{\partial F^{2}}{\partial x^{h}})$. The horizontal and vertical subspaces have the corresponding bases $\{\frac{\delta}{\delta {x^i}},\frac{\partial}{\partial y^{i}}\}$, which are related to the typical bases of $TM$ $\{\frac{\partial}{\partial x^{i}},\frac{\partial}{\partial y^{i}}\}$, by $\frac{\delta}{\delta {x^i}}:=\frac{\partial}{\partial x^{i}}-G_{i}^{j}\frac{\partial}{\partial y^{j}}$. The dual bases of the former basis are denoted by $\{dx^{i},\delta y^{i}\}$, where $\delta y^{i}:=dy^{i}+G_{j}^{i}dx^{j}$. The 1-form of Cartan connection in these bases are given by $\omega^{i}_{j}=\Gamma^{i}_{jk}dx^{k}+C^{i}_{jk}\delta y^{k}$, where $\Gamma^{i}_{jk}=\frac{1}{2}g^{ih}(\delta_{j}g_{hk}+\delta_{k}g_{jh}-\delta_{h}g_{jk})$, $C^{i}_{jk}=\frac{1}{2}g^{ih}\dot{\partial}_{h}g_{jk}$, $\delta_{k}=\frac{\delta}{\delta x^{k}}$ and $\dot{\partial}_{k}=\frac{\partial}{\partial y^{k}}$. In local coordinates, coefficients of the Cartan connection $\nabla$ are given by
$$\nabla_{k}\dot{\partial}_{j}=\Gamma^{i}_{jk}\dot{\partial}_{i},\hspace{0.4cm}
\dot{\nabla}_{k}\dot{\partial} _{j}=C^{i}_{jk}\dot{\partial}_{i},\hspace{0.4cm}
\nabla_{k}\delta_{j}=\Gamma^{i}_{jk}\delta_{i},\hspace{0.4cm} \dot{\nabla}_{k}\delta_{j}=C^{i}_{jk}\delta_{k},$$
where, $\nabla_{k}:=\nabla_{\frac{\delta}{\delta x^k}}$ and $\dot{\nabla}_{k}:=\nabla_{\frac{\partial}{\partial y^k}}$, see \cite{BCS}.
The components of Cartan horizontal and vertical covariant derivatives of a Finslerian $(1,2)$ tensor field $S$ on $\pi^{*}TM$ with the components $(S^{i}_{jk}(x,y))$ on $TM$ are given by
\begin{align}\label{COV}
&\nabla_{l}S^{i}_{jk}:= \delta_{l}S^{i}_{jk}-S^{i}_{s k}\Gamma^{s}_{jl}-S^{i}_{js}\Gamma^{s}_{kl}+S^{s}_{jk}\Gamma^{i}_{s l},\nonumber\\
&\dot{\nabla}_{l}S^{i}_{jk}:=\dot{\partial}_{l}S^{i}_{jk}-S^{i}_{s k}C^{s}_{jl}-S^{i}_{js}C^{s}_{kl}+S^{s}_{jk}C^{i}_{s l},
\end{align}
respectively. The horizontal and vertical metric compatibility of  Cartan connection  in local coordinates are written $\nabla_{l}g_{jk}=0$ and $\dot{\nabla}_{l}g_{jk}=0$, respectively. The horizontal covariant derivative of a $(0,2)$ tensor $T$ is written as follows
\begin{equation}\label{codev}
(\nabla_{H\hat{X}}T)(Y,Z)=\nabla_{H\hat{X}}T(Y,Z)-T(\nabla_{H\hat{X}}Y,Z)-T(Y,\nabla_{H\hat{X}}Z).
\end{equation}
\subsection{The $hh$-curvature tensor of Cartan connection}
Let us consider the horizontal curvature operator
\begin{equation*}
R(X,Y)Z:=\Omega(H\hat{X},H\hat{Y})Z=\nabla_{H\hat{X}}\nabla_{H\hat{Y}}Z
-\nabla_{H\hat{Y}}\nabla_{H\hat{X}}Z
-\nabla_{[H\hat{X},H\hat{Y}]}Z,
\end{equation*}
where, $X,Y,Z \in\Gamma(\pi^{*}TM)$ and $\hat{X},\hat{Y}\in{\cal X}(TM_0)$.
The \emph{$hh$-curvature tensor} of Cartan connection is defined by $R(W,Z,X,Y):=g(R(X,Y)Z,W)$. Replacing $W$ with the local frame $\{e_k\}_{k=1}^n$ we get
\begin{align}\label{9+1}
R(X,Y)Z= \sum_{k=1}^n R(e_k,Z,X,Y)e_k.
\end{align}
One can check that the $hh$-curvature of Cartan connection is skew-symmetric with respect to the first two vector fields as well as the last two vector fields, see \cite{HAZ}, page 43. That is,
\begin{align*}
&R(X,Y,Z,W)=-R(Y,X,Z,W),\\
&R(X,Y,Z,W)=-R(X,Y,W,Z).
\end{align*}
In a local coordinate system we have
\begin{equation*}
R(\partial_{i},\partial_{j})\partial_{k}=\Omega(\delta_{i},\delta_{j})\partial_{k}=R_{\,\,kij}^{h}\partial_{h}.
\end{equation*}
Recall that the upper index is placed in the \emph{first} position, that is
\begin{equation*}
R_{tkij}:=g_{ht}R^{h}_{\,\,kij}=g(R(\partial_{i},\partial_{j})\partial_{k},\partial_{t}).
\end{equation*}
The components of Cartan $hh$-curvature tensor are given by
\begin{equation} \label{77}
R^{h}_{\,\,kij}=\delta_{i}\Gamma^{h}_{\,jk}-\delta_{j}\Gamma^{h}_{\,ik}+
\Gamma^{l}_{\,jk}\Gamma^{h}_{\,il}-\Gamma^{l}_{\,ik}\Gamma^{h}_{\,jl}
+R^{l}_{\,\,ij}C^{h}_{\,lk},
\end{equation}
where, $R^{l}_{\,\,ij}=y^{p} R^{l}_{\,\,pij}$. The \emph{reduced $hh$-curvature} is defined by
\begin{equation*}
R(X,Y,Z,W):=R(X,l,Z,l),
\end{equation*}
where, $l:=\frac{y^i}{F}\frac{\partial}{\partial x^i}$ is the distinguished global section. The reduced $hh$-curvature is a connection free tensor called also Riemann curvature by certain authors. In the local coordinates the reduced $hh$-curvature is given by $R^{i}_{\,\,k}:=\frac{1}{F^2}y^{j}R^{i}_{\,\,jkm}y^{m}$ which are entirely expressed in terms of $x$ and $y$ derivatives of spray coefficients $G^{i}$ as follows 
\begin{equation} \label{18}
R^{i}_{\,\,k}:=-\frac{1}{F^2}(2\frac{\partial G^{i}}{\partial x^{k}}-\frac{\partial^{2}G^{i}}{\partial x^{j}\partial y^{k}}y^{j}+2G^{j}\frac{\partial^{2}G^{i}}{\partial y^{j}\partial y^{k}}-\frac{\partial G^{i}}{\partial y^{j}}\frac{\partial G^{j}}{\partial y^{k}}).
\end{equation}
Note that the components of reduced $hh$-curvature tensor in (\ref{18}) are different in a sign by that in \cite{BCS} page 66, using Chern connection.
\subsection{Flag curvature and Ricci scalar}
Consider the vector field $l$, called the flagpole, and the unit vector $V=V^{i}\frac{\partial}{\partial x^{i}}\in\Gamma(\pi^{*}TM)$, called the transverse edge, which is perpendicular to the flagpole, the \emph{flag curvature}
is defined by
\begin{equation*}
K(x,y,l\wedge V):=V^{j}(l^{i}R_{jikl}l^{l})V^{k}=:V^{j}R_{jk}V^{k}.
\end{equation*}
If the transverse edge $V$ is orthogonal to the flagpole but not necessarily of unit length, then
\begin{equation}\label{flag}
K(x,y,l\wedge V)=\frac{V^{j}R_{jk}V^{k}}{g(V,V)}.
\end{equation}
The case in which $V$ is neither of unit length nor orthogonal to $l$ is treated in page 191, \cite{BCS}. The \emph{Ricci scalar} is defined as trace of the flag curvature i.e.
\begin{equation}\label{Ric}
\mathcal{R}ic:=\sum_{\alpha=1}^{n-1}K(x,y,l\wedge e_{\alpha}),
\end{equation}
where, $\{e_{1},...,e_{n-1},l\}$ is considered as a $g$-orthonormal basis for $T_{x}M$. Equivalently,
\begin{equation*}
\mathcal{R}ic=g^{ik}R_{ik}=R^{i}_{\,\,i},
\end{equation*}
where, $R^{i}_{\,\,k}$ are defined by (\ref{18}).
\subsection{A Riemannian connection on the indicatrix}\label{IND}
For a fixed point $x_{0}\in M$, the fiber $\pi^{-1}(x_0)=T_{x_0}M$ is a submanifold of $TM_0$ with the Riemannian metric $\tilde{g}(X,Y)=g_{ij}(x_{0},y)dy^{i}dy^{j}(X,Y)$ determined by the vertical part of Sasakian metric on $TM$ where, $X,Y\in V_{z}T_{x_0}M$. The hyper-surface $S_{x_0}M=\{y\in T_{x_0}M:F(x_{0},y)=1\}$ of $T_{x_0}M$ is called \emph{indicatrix} in $x_0\in M$. On the other hand a hyper-surface $S_{x_0}M$ can be expressed in local coordinates by the coordinate functions
\begin{equation*}
y^{i}=y^{i}(t^{\alpha}),
\end{equation*}
where, the Greek letters $\alpha,\beta,\gamma,...$ run over the range $1,...,n-1$ and the Latin letters $i,j,k,...$ run over the range $1,...,n$.
Let $f$ be a real function defined on $S_{x_0}M$. By chain rule we have $df(y(t))=\partial_{\alpha}fdt^{\alpha}$ where,
\begin{equation}\label{basis}
\partial_{\alpha}=y^{i}_{\alpha}F\dot{\partial}_{i},\quad y^{i}_{\alpha}=\frac{\partial y^i}{\partial t^\alpha}.
\end{equation}
Hence $\partial_{\alpha}$, define $(n-1)$ tangent vectors on $S_{x_0}M$. The induced Riemannian metric tensor $g_{\alpha\beta}$ on $S_{x_0}M$ is given by
\begin{align*}
g_{\alpha\beta}=g_{ij}y^{i}_{\alpha}y^{j}_{\beta},
\end{align*}
where, $g_{ij}(x_{0},y)$ are the components of Riemannian metric tensor on $T_{x_0}M$. Let $\dot{y}=y^j\dot{\partial}_{j}$ be a vector field tangent to the fiber through $z=\pi^{-1}(x_{0})$. Partial derivatives of $F^{2}(x,y)=1$ with respect to $y^{i}$, yields
\begin{equation}\label{orthogonal}
g_{ij}y^jy^{i}_{\alpha}=g(\partial_{\alpha},\dot{y})=0.
\end{equation}
Therefore, $\dot{y}$ is normal to the $(n-1)$ tangent vectors $y^{i}_{\alpha}$ of $S_{x_0}M$ and hence the pair $(y^{i}_{\alpha},\dot{y})$ defines $n$ linearly independent tangent vector fields on $T_{x_0}M$. We denote by $\dot{D}_{\dot{\partial}_k}$ the corresponding Riemannian covariant derivative on $(T_{x_{0}}M,\tilde{g})$, where the coefficients are given by
\begin{equation*}
\dot{D}_{\dot{\partial}_k}\dot{\partial}_j=C^{i}_{jk}(x_{0},y)\dot{\partial}_{i}.
\end{equation*}
Let $\dot{\nabla}$ be the induced connection on $(S_{x_0}M,g_{\alpha\beta})$. Relation between $\dot{D}$ and $\dot{\nabla}$ is given by the Gaussian formula
\begin{equation*}
\dot{D}_{Y}X=\dot{\nabla}_{Y}X-\tilde{g}(X,Y)\dot{y},
\end{equation*}
where, $X,Y\in T_{z}(S_{x_0}M)$. Replacing $X$ and $Y$ by the basis fields $\partial_{\alpha}$ and $\partial_{\beta}$ yields
\begin{equation}\label{9}
\dot{\nabla}_{\beta}y^{i}_{\alpha}=-A^{i}_{jk}y^{j}_{\alpha}y^{k}_{\beta}-g_{\alpha\beta}y^{i},
\end{equation}
where, $A^{i}_{jk}=FC^{i}_{jk}$, see \cite{HAZ}, pages 147-149.
\subsection{Local basis on the unitary sphere bundle $SM$}
Consider the sphere bundle $SM:=TM/\sim$ as a quotient space, where the equivalent relation is defined by $y\sim y'$ if and only if $y=\lambda y'$ for some $\lambda>0$. Given any $(x,y)\in TM$, we shall denote its equivalence class as a point in $SM$ by $(x,[y])\in SM$.
The natural projection $p:SM\longrightarrow M$ pulls back the tangent bundle $TM$ to an n-dimensional vector bundle $p^{*}TM$ over the $2n-1$ dimensional base $SM$.
Given local coordinates $(x^i)$ on $M$, we shall economize on notation and regard the corresponding collections $\{\frac{\partial}{\partial x^i}\}$, $\{dx^i\}$ as local bases for the pull back bundle $p^{*}TM$ and its dual $p^{*}T^{*}M$, respectively.

Let $\{e_{a}=u_{a}^{i}\frac{\partial}{\partial x^i}\}$ be a local orthonormal frame for $p^{*}TM$ and $\{\omega^a=v^{a}_{i}dx^i\}$ its co-frame, where $\omega^{a}(e_{b})=\delta^{a}_{b}$. Clearly we have $e_{n}:=l$, where $l=\frac{y^i}{F}\frac{\partial}{\partial x^i}$ is the distinguished global section and $\omega^{n}=\frac{\partial F}{\partial y^i}dx^{i}$.
Also we have $\frac{\partial}{\partial x^{i}}=v^{a}_{i}e_{a}$ and $dx^{i}=u^{i}_{a}\omega^{a}$ where, relation between $(u^i_a)$ and $(v^a_i)$ are given by $v^a_iu_b^i=\delta^a_b$ and $u^i_av^a_j=\delta^i_j$.
For convenience, we shall also regard the $e_{a}$'s and $\omega^a$'s as local vector fields and 1-forms, respectively on $SM$, see \cite{BAOS}. Let us define
\begin{eqnarray*}
&&\hat{e}_{a}=u^{i}_{a}\frac{\delta}{\delta x^{i}},\quad
\hat{e}_{n+\alpha}=u^{i}_{\alpha}F\frac{\partial}{\partial y^i},\\
&&\omega^{a}=v^{a}_{i}dx^{i},\quad
\omega^{n+\alpha}=v^{\alpha}_{i}\frac{\delta y^i}{F}.
\end{eqnarray*}
It can be shown that $\{\hat{e}_{a},\hat{e}_{n+\alpha}\}$ and $\{\omega^{a},\omega^{n+\alpha}\}$ are local basis for the tangent bundle $TSM$ and the cotangent bundle $T^{*}SM$, respectively, where the Latin indices $a,b,...$ run over the range $1,...,n$ and the Greek indices run over the range $1,...,n-1$. Tangent vectors on $SM$ which are annihilated by all $\{\omega^{n+\alpha}\}$'s form the horizontal sub-bundle $HSM$ of $TSM$. The fibers of $HSM$ are $n$-dimensional and $\{\hat{e}_{a}\}$ is a local basis for the fibers of $HSM$. On the other hand, let $VSM:=\cup_{x\in M}T(S_{x}M)$ be the vertical sub-bundle of $TSM$. Its fibers are $n-1$ dimensional and $\{\hat{e}_{n+\alpha}\}$ is a local basis for the fibers of $VSM$. Here, $\hat{e}_{n+\alpha}$ coincide with $\partial_{\alpha}$ previously mentioned in Subsection \ref{IND}. The decomposition $TSM=HSM\oplus VSM$ holds well because $HSM$ and $VSM$ are directly summed, see \cite{BAOS}.

\subsection{ Ricci tensors and Ricci flows in Finsler space}
There are several well known definitions for Ricci tensor in Finsler geometry. For instance, H. Akbar-Zadeh has considered two Ricci tensors on Finsler manifolds in his works namely, one is defined by $ Ric_{ij}:=[\frac{1}{2}F^{2}\mathcal{R}ic]_{y^{i}y^{j}}$ and another by
 $Rc_{ij}:=\frac{1}{2} (\textsf{R}_{ij}+\textsf{R}_{ji})$, where $ \textsf{R}_{ij}$ is the trace of $hh$-curvature of Cartan connection defined by $ \textsf{R}_{ij}=R^l_{\ ilj}$.
 D. Bao based on the first definition of Ricci tensor has considered the following Ricci flow in Finsler geometry,
\begin{equation}\label{20002}
\frac{\partial}{\partial t}g_{jk}(t)=-2Ric_{jk},\hspace{0.6cm}g_{(t=0)}=g_{0},
\end{equation}
where, $g_{jk}(t)$ is a family of Finslerian metrics defined on $\pi^{*}TM\times[0,T)$.
Contracting (\ref{20002}) with $y^{j}y^{k}$, via Euler's theorem, leads to $\frac{\partial}{\partial t}F^{2}=-2F^{2}\mathcal{R}ic$. That is,
\begin{equation} \label{20} 
\frac{\partial}{\partial t}\log F(t)=-\mathcal{R}ic,\hspace{0.6cm}F_{(t=0)}=F_{0},
\end{equation}
where, $F_{0}$ is the initial Finsler structure, see \cite{Bao}. It can be easily verified that (\ref{20002}) and (\ref{20}) are equivalent. This Ricci flow is used in \cite{BKS, BY, YB2, YB1, SL}. Here and everywhere in the present work  we consider the first Akbar-Zadeh's definition of Ricci tensor and the related  Ricci flow studied by D. Bao.

One of the advantages of the Ricci quantity used here is its independence on the choice of Cartan, Berwald or Chern(Rund) connections. Another preference of this Ricci tensor is the  parabolic form of its Ricci scalar's evolution in the sense given in Proposition \ref{asli3}.

We say that the Ricci tensor has a \emph{uniform bound} if there is a constant $K$ such that $\parallel Ric_{(x,y,t)}\parallel_{g(t)}\leq K$, where $\parallel .\parallel_{g(t)}$ is the norm defined by $g(t)$.
\subsection{Statement of the maximum principle}
We recall here the weak maximum principle states that the extremum of solutions to elliptic equations are dominated by their extremum on the boundary, more intuitively we have the following theorem.
\begin{propriete}\label{maximum principle}
\cite{TAPP} (Weak maximum principle for scalars). Let $M$ be a closed manifold. Assume, for $t\in[0,T]$, where $0<T<\infty$, that $g(t)$ is a smooth family of metrics on $M$, and $X(t)$ is a smooth family of vector fields on $M$. Let $f:\mathbb{R}\times[0,T]\longrightarrow \mathbb{R}$ be a smooth function. Suppose that $u\in C^{\infty}(M\times[0,T],\mathbb{R})$ solves
\begin{equation*}
\frac{\partial u}{\partial t}\leq\Delta_{g(t)}u+<X(t),\nabla u>+f(u,t).
\end{equation*}
Suppose further that $\phi:[0,T]\longrightarrow \mathbb{R}$ solves
\begin{equation*}
\left\{
\begin{array}{l}
\frac{d \phi}{d t}=f(\phi(t),t),\cr
\phi(0)=\alpha\in \mathbb{R}.
\end{array}
\right.
\end{equation*}
If $u(\,.\,,0)\leq\alpha$, then $u(\,.\,,t)\leq\phi(t)$ for all $t\in [0,T].$
\end{propriete}
By applying this result when the signs of $u$, $\phi$ and $\alpha$ are reversed and $f$ is appropriately modified, we find the following modification:
\begin{Corollary}\label{minimum principle}
\cite{TAPP} (Weak minimum principle). Theorem \ref{maximum principle} also holds with the sense of all three inequalities reversed, that is, replacing all three instances of $\leq$ by $\geq$.
\end{Corollary}
\section{Evolution of the reduced curvature tensor}
In this section, we derive evolution equation for the reduced $hh$-curvature $R(X,Z)$ along the Ricci flow and show that if $(M,F(0))$ has positive reduced $hh$-curvature at the initial time, namely, $R_{g(0)}> 0$, then $(M,F(t))$ has positive reduced $hh$-curvature $R_{g(t)}>0$ for all $t\in[0,T)$. 
Let $X$ and $Y$ be two fixed sections of the pulled back bundle $\pi^{*}TM$ in the sense that $X$ and $Y$ are independent of $t$ and define $A(X,Y):=\frac{\partial}{\partial t}(\nabla_{H\hat{X}}Y)$. Now we are in a position to prove the following proposition.
\begin{prop}
Let $Z,X\in\Gamma(\pi^{*}TM)$ be two fixed vector fields on $TM_{0}$. Then
\begin{equation}\label{eq.1}
\frac{\partial}{\partial t}
(F^{2}R(Z,X))=-2\sum_{k=1}^{n}F^{2}R(e_{k},X)Ric(e_{k},Z),
\end{equation}
where, $R(Z,X)=\frac{1}{F^{2}}R(Z,u,X,u)$ is the reduced $hh$-curvature and $\big\{e_{k}\big\}_{k=1}^{n}$ is an orthonormal basis for $\pi^{*}TM$.
\end{prop}
\begin{proof}
Let $W,Z\in\Gamma(\pi^{*}TM)$ and $\hat{X},\hat{Y}\in{\cal X}(TM_{0})$ be fixed vector fields on $TM$. By definition of the $hh$-curvature tensor and the equations (\ref{9+1}) and (\ref{20002})  we have
\begin{align*}
\frac{\partial}{\partial t}(R(Z,W,X,Y))&=\frac{\partial}{\partial t}(g(R(X,Y)W,Z))\\
&=(\frac{\partial}{\partial t} g)(R(X,Y)W,Z)+g(\frac{\partial}{\partial t}R(X,Y)W,Z)\\
&=-2Ric(\sum_{k=1}^{n}R(e_{k},W,X,Y)e_{k},Z)\\
&\ \ +g\Big(\frac{\partial}{\partial t}(\nabla_{H\hat{X}}\nabla_{H\hat{Y}}W-\nabla_{H\hat{Y}}\nabla_
{H\hat{X}}W-\nabla_{[H\hat{X},H\hat{Y}]}W),Z\Big ).
\end{align*}
Using the notation
 $A(X,Y)=\frac{\partial}{\partial t}(\nabla_{H\hat{X}}Y)$ leads
\begin{align*}
\frac{\partial}{\partial t}(R(Z,W,X,Y))&=-2\sum_{k=1}^{n}R(e_{k},W,X,Y)Ric(e_{k},Z)+g\Big(A(X,\nabla_{H\hat{Y}}W),Z\Big)\\
&\quad+g\Big(\nabla_{H\hat{X}}(A(Y,W)),Z\Big)-g\Big(A(Y,\nabla_{H\hat{X}}W),Z\Big)\\
&\quad-g\Big(\nabla_{H\hat{Y}}(A(X,W)),Z\Big)-g\Big(A(\rho[H\hat{X},H\hat{Y}],W),Z\Big).
\end{align*}
By means of the horizontal torsion freeness $S(X,Y)=0$, we have $\nabla_{H\hat{X}}W-\nabla_{H\hat{W}}X =\rho[H\hat{X},H\hat{W}]$. Applying the horizontal covariant derivative (\ref{codev}) to $A$, the above equation leads to
\begin{align}\label{eq.2}
\frac{\partial}{\partial t}(R(Z,W,X,Y))&=-2\sum_{k=1}^{n}R(e_{k},W,X,Y)Ric(e_{k},Z)+g\Big((\nabla_{H\hat{X}}A)(Y,W),Z\Big)\nonumber\\
&\quad-g\Big((\nabla_{H\hat{Y}}A)(X,W),Z\Big).
\end{align}
Let $u=y^{i}\frac{\partial}{\partial x^{i}}$ be the canonical section. Since its horizontal derivative vanishes, namely $\nabla_{H\hat X}u=0$, we have
\begin{equation*}
g((\nabla_{H\hat{X}}A)(u,u),Z)=g((\nabla_{\hat{u}}A)(X,u),Z)=0,
\end{equation*}
where, $\hat{u}=y^{i}\frac{\delta}{\delta x^{i}}$. Therefore, letting $Y=W=u$ and using $R(Z,u,X,u)=F^2R(Z,X)$ the equation (\ref{eq.2}) reduces to
\begin{equation*}
\frac{\partial}{\partial t}(F^{2}R(Z,X))=-2\sum_{k=1}^{n}F^{2}R(e_{k},X)Ric(e_{k},Z).
\end{equation*}
This completes the proof.
\end{proof}
If we put $\bar{R}(Z,X):=F^{2}R(Z,X)$, then (\ref{eq.1}) reads
\begin{equation}\label{eq.3}
\frac{\partial}{\partial t}\bar{R}(Z,X)=-2\sum_{k=1}^{n}\bar{R}(e_{k},X)Ric(e_{k},Z).
\end{equation}
\begin{prop}\label{asli1}
Let $(M^{n},F(t))$ be a family of solutions to the Finslerian Ricci flow. If there is a constant $K$ such that $\parallel Ric\parallel_{g(t)}\leq K$ on the time interval $[0,T)$, and the reduced $hh$-curvature $R_{g(0)}$ of $F(0)$ is  positive that is, $R_{g(0)}(V,V)>0$ for all $V\in\Gamma(\pi^{*}TM)$ perpendicular to the distinguished global section $l$, then there exists a positive constant $C(n)$ such that
\begin{equation*}
e^{-2KCT}\bar{R}_{(x,y,0)}(V,V)\leq\bar{R}_{(x,y,t)}(V,V)\leq e^{2KCT}\bar{R}_{(x,y,0)}(V,V),
\end{equation*}
for all $(x,y)\in TM$ and $t\in[0,T)$.
\end{prop}
\begin{proof}
Let $(x,y)\in TM$, $t_{0}\in[0,T)$ and $V\in\Gamma(\pi^{*}TM)$ be a nonzero arbitrary section perpendicular to the distinguished global section $l:=\frac{y^i}{F}\frac{\partial}{\partial x^i}$. We have
\begin{align}\label{eq.3+1}
\parallel \log(\frac{\bar{R}_{(x,y,t_{0})}(V,V)}{\bar{R}_{(x,y,0)}(V,V)})\parallel&=\parallel\int_{0}^{t_{0}}\frac{\partial}{\partial t}[\log \bar{R}_{(x,y,t)}(V,V)]dt\parallel\nn\\
&=\parallel\int_{0}^{t_{0}}\frac{\frac{\partial}{\partial t}\bar{R}_{(x,y,t)}(V,V)}{\bar{R}_{(x,y,t)}(V,V)}dt\parallel.
\end{align}
By means of (\ref{eq.3})  we have
\begin{equation*}
\parallel\int_{0}^{t_{0}}\frac{\frac{\partial}{\partial t}\bar{R}_{(x,y,t)}(V,V)}{\bar{R}_{(x,y,t)}(V,V)}dt\parallel=
\parallel\int_{0}^{t_{0}}\frac{-2\sum_{k=1}^{n}\bar{R}_{(x,y,t)}(e_{k},V)Ric_{(x,y,t)}(e_{k},V)}{\bar{R}_{(x,y,t)}(V,V)}dt\parallel.
\end{equation*}
Therefore, (\ref{eq.3+1}) becomes
\begin{align*}
\parallel \log(\frac{\bar{R}_{(x,y,t_{0})}(V,V)}{\bar{R}_{(x,y,0)}(V,V)})\parallel&=\parallel\int_{0}^{t_{0}}\frac{-2\sum_{k=1}^{n}\bar{R}_{(x,y,t)}(e_{k},V)
Ric_{(x,y,t)}(e_{k},V)}{\bar{R}_{(x,y,t)}(V,V)}dt\parallel\\
&=\parallel\int_{0}^{t_{0}}\frac{2<\bar{R}_{(x,y,t)}(V),Ric_{(x,y,t)}(V)>}{\bar{R}_{(x,y,t)}(V,V)}dt\parallel\\
&\leq\int_{0}^{t_{0}}\parallel\frac{2<\bar{R}_{(x,y,t)}(V),Ric_{(x,y,t)}(V)>}{\bar{R}_{(x,y,t)}(V,V)} \parallel dt.
\end{align*}
By means of Cauchy-Schwarz inequality we have
\begin{equation*}
\parallel <\bar{R}_{(x,y,t)}(V),Ric_{(x,y,t)}(V)> \parallel\leq \parallel\bar{R}_{(x,y,t)}(V) \parallel \parallel  Ric_{(x,y,t)}(V)\parallel.
\end{equation*}
Therefore, we obtain
\begin{align}\label{eq.4}
\parallel \log(\frac{\bar{R}_{(x,y,t_{0})}(V,V)}{\bar{R}_{(x,y,0)}(V,V)})\parallel
\leq\int_{0}^{t_{0}}2\frac{\parallel\bar{R}_{(x,y,t)}(V) \parallel \parallel  Ric_{(x,y,t)}(V)\parallel}{\parallel \bar{R}_{(x,y,t)}(V,V) \parallel}dt.
\end{align}
There exists a positive constant $C$, depending only on $n$ such that
\begin{align}\label{eq.5}
\hspace{-0.5cm}\parallel\bar{R}_{(x,y,t)}(V) \parallel \parallel  Ric_{(x,y,t)}(V)\parallel\leq C\parallel \bar{R}_{(x,y,t)}(V,V) \parallel  \parallel Ric_{(x,y,t)}(V,V) \parallel.
\end{align}
By means of (\ref{eq.4}) and (\ref{eq.5}) and using the fact that $\parallel T(U,U)\parallel\leq\parallel T\parallel_{g(t)}$ for the any 2-tensor $T$ and the unit vector $U$, we have
\begin{align*}
\parallel \log(\frac{\bar{R}_{(x,y,t_{0})}(V,V)}{\bar{R}_{(x,y,0)}(V,V)})\parallel&\leq\int_{0}^{t_{0}}2C\parallel Ric_{(x,y,t)}(V,V) \parallel dt\\
&\leq\int_{0}^{t_{0}}2C\parallel Ric_{(x,y,t)} \parallel_{g(t)} dt\\
&\leq\int_{0}^{t_{0}}2CKdt\\
&\leq2CKT.
\end{align*}
By assumption $R_{(x,y,0)}(V,V)>0$ and hence $\bar R_{(x,y,0)}(V,V)>0$. Therefore, the uniform bound on $\bar{R}_{(x,y,t)}(V,V)$ follows from exponentiation, namely,
\begin{equation*}
e^{-2KCT}\bar{R}_{(x,y,0)}(V,V)\leq\bar{R}_{(x,y,t)}(V,V)\leq e^{2KCT}\bar{R}_{(x,y,0)}(V,V),
\end{equation*}
for all $(x,y)\in TM$ and $t\in[0,T)$. This completes the proof.
\end{proof}
Proposition \ref{asli1} implies that
if $(M^{n},F(t))$ is a family of solutions to the Finslerian Ricci flow satisfying a uniform Ricci tensor bound on a finite time interval $[0,T)$, then positive reduced $hh$-curvature is preserved under the Ricci flow. More precisely,
\begin{prop}\label{asli2}
Let $(M^{n},F(t))$ be a family of solutions to the Finslerian Ricci flow with $F(0)=F_{0}$. If there is a constant $K$ such that $\parallel Ric\parallel_{g(t)}\leq K$ on the time interval $[0,T)$ and the reduced $hh$-curvature $R_{g(0)}$ of $F(0)$ is  positive, that is, $R_{g(0)}(V,V)>0$ for all $V\in\Gamma(\pi^{*}TM)$ perpendicular to the distinguished global section $l$, then the reduced $hh$-curvature $R_{g(t)}$ of $F(t)$ remains positive in short time, namely, $R_{g(t)}(V,V)>0$ for all $t\in[0,T)$.
\end{prop}

{\bf Proof of Theorem \ref{asli4}.} By assumption $(M,F(0))$ has positive flag curvature. Definition of the flag curvature (\ref{flag}) implies that $R_{g_{0}}>0$. By means of Proposition \ref{asli2}, $R_{g(t)}>0$ for all $t\in[0,T)$. Using the definition of the flag curvature (\ref{flag}) once more shows that $F(t)$ has positive flag curvature, as long as the solution exists. By means of this fact and definition of the Ricci scalar (\ref{Ric}) we have $\mathcal{R}ic_{g(t)}>0$ for all $t\in[0,T)$. This completes the proof of Theorem \ref{asli4}.\hspace{\stretch{1}}$\Box$
\section{Evolution of the Ricci scalar $\mathcal{R}ic$}
\begin{prop}\label{asli3}
The Ricci scalar of $g(t)$ satisfies the evolution equation
\begin{equation}\label{eq.6}
\frac{\partial}{\partial t}\mathcal{R}ic=-F^{2}R^{ij}\frac{\partial^{2}}{\partial y^{i}\partial y^{i}}\mathcal{R}ic.
\end{equation}
\end{prop}
\begin{proof}
By means of (\ref{eq.1}) and taking the trace over $Z$ and $X$ we obtain
\begin{equation}\label{eq.7}
\frac{\partial}{\partial t}(\sum_{l=1}^{n}F^{2}R(e_{l},e_{l}))=-2F^{2}\sum_{k,l=1}^{n}
R(e_{k},e_{l})Ric(e_{k},e_{l}).
\end{equation}
In the natural basis, (\ref{eq.7}) is written
\begin{equation}\label{eq.8}
\frac{\partial}{\partial t}(F^{2}\mathcal{R}ic)=-2F^{2}R^{ij}Ric_{ij}.
\end{equation}
By means of chain rule and definition of Ricci tensor, (\ref{eq.8}) is written as follows
\begin{equation*}
\frac{\partial}{\partial t}\mathcal{R}ic=-F^{2}R^{ij}\frac{\partial^{2}}{\partial y^{i}\partial y^{i}}\mathcal{R}ic-2(tr_{g}R)\mathcal{R}ic+2\mathcal{R}ic^{2}.
\end{equation*}
Since $tr_{g}R=\mathcal{R}ic$, we have
\begin{equation*}
\frac{\partial}{\partial t}\mathcal{R}ic=-F^{2}R^{ij}\frac{\partial^{2}}{\partial y^{i}\partial y^{i}}\mathcal{R}ic.
\end{equation*}
This completes the proof.
\end{proof}
In the remainder of this section, we discuss one implication of Proposition \ref{asli3}.

{\bf Proof of Theorem \ref{asli5}.} By means of Proposition \ref{asli3}, the Ricci scalar satisfies the evolution equation (\ref{eq.6}). One can rewrite (\ref{eq.6}) with respect to the basis of $TSM$. By means of (\ref{basis}) we have
\begin{equation*}
\partial_{\beta}\mathcal{R}ic=Fy^{j}_{\beta}\frac{\partial\mathcal{R}ic}{\partial y^{j}}.
\end{equation*}
The vertical covariant derivative leads
\begin{align}\label{eq.8+1}
\dot{\nabla}_{\alpha}\partial_{\beta}\mathcal{R}ic&=\dot{\nabla}_{\alpha}
(Fy^{j}_{\beta}\frac{\partial\mathcal{R}ic}{\partial y^{j}})\nonumber\\
&=(\dot{\nabla}_{\alpha}F)y^{j}_{\beta}\frac{\partial\mathcal{R}ic}{\partial y^{j}}+F(\dot{\nabla}_{\alpha}y^{j}_{\beta})\frac{\partial\mathcal{R}ic}{\partial y^{j}}+Fy^{j}_{\beta}(\dot{\nabla}_{\alpha}\frac{\partial\mathcal{R}ic}{\partial y^{j}}).
\end{align}
On the other hand
\begin{equation}\label{eq.8+2}
\dot{\nabla}_{\alpha}F=\dot{\nabla}_{Fy^{i}_{\alpha}\frac{\partial}{\partial y^i}}F=Fy^{i}_{\alpha}F_{y^i}=Fy^{i}_{\alpha}l_{i}
=g_{ij}y^jy^{i}_{\alpha}.
\end{equation}
By means of (\ref{orthogonal}) and (\ref{eq.8+2}) we have $\dot{\nabla}_{\alpha}F=0$. Using (\ref{9}), equation (\ref{eq.8+1}) becomes
\begin{align}\label{eq.9+1}
\dot{\nabla}_{\alpha}\partial_{\beta}\mathcal{R}ic&=F(-A^{j}_{kl}y^{k}_{\alpha}y^{l}_{\beta}-g_{\alpha\beta}y^{j})\frac{\partial\mathcal{R}ic}{\partial y^{j}}+Fy^{j}_{\beta}(\dot{\nabla}_{\alpha}\frac{\partial\mathcal{R}ic}{\partial y^{j}})\nonumber\\
&=-FA^{j}_{kl}y^{k}_{\alpha}y^{l}_{\beta}\frac{\partial\mathcal{R}ic}{\partial y^{j}}+Fy^{j}_{\beta}(\dot{\nabla}_{Fy^{i}_{\alpha}\frac{\partial}{\partial y^{i}}}\frac{\partial\mathcal{R}ic}{\partial y^{j}})\nonumber\\
&=-FA^{j}_{kl}y^{k}_{\alpha}y^{l}_{\beta}\frac{\partial\mathcal{R}ic}{\partial y^{j}}+F^{2}y^{i}_{\alpha}y^{j}_{\beta}\dot{\nabla}_{
\frac{\partial}{\partial y^{i}}}\frac{\partial\mathcal{R}ic}{\partial y^{j}}.
\end{align}
By the vertical covariant derivative (\ref{COV}), equation (\ref{eq.9+1}) is written
\begin{align}\label{eq.9+2}
\dot{\nabla}_{\alpha}\partial_{\beta}\mathcal{R}ic&=-FA^{j}_{kl}y^{k}_{\alpha}y^{l}_{\beta}\frac{\partial\mathcal{R}ic}{\partial y^{j}}+F^{2}y^{i}_{\alpha}y^{j}_{\beta}(\frac{\partial^{2}\mathcal{R}ic}{\partial y^{i}\partial y^{j}}-C^{k}_{ij}\frac{\partial \mathcal{R}ic}{\partial y^{k}})\nonumber\\
&=F^{2}y^{i}_{\alpha}y^{j}_{\beta}\frac{\partial^{2}\mathcal{R}ic}{\partial y^{i}\partial y^{j}}-2FA^{k}_{ij}y^{i}_{\alpha}y^{j}_{\beta}\frac{\partial \mathcal{R}ic}{\partial y^{k}}.
\end{align}
Converting (\ref{eq.9+2}) in $R^{\alpha\beta}=F^{-2}R^{ij}y^{\alpha}_{i}y^{\beta}_{j}$ yields
\begin{equation}\label{eq.9+3}
R^{\alpha\beta}\dot{\nabla}_{\alpha}\partial_{\beta}\mathcal{R}ic=R^{ij}\frac{\partial^{2}\mathcal{R}ic}{\partial y^{i}\partial y^{j}}-2F^{-1}A^{k}_{ij}R^{ij}\frac{\partial \mathcal{R}ic}{\partial y^{k}}.
\end{equation}
Using (\ref{basis}) we have $\dot{\partial}_{k}=F^{-1}y^{\lambda}_{k}  \partial_{\lambda}$ and from which $\frac{\partial \mathcal{R}ic}{\partial y^{k}}=F^{-1}y^{\lambda}_{k}\partial_{\lambda}\mathcal{R}ic$. Hence, replacing in (\ref{eq.9+3}) we obtain
\begin{equation*}
R^{ij}\frac{\partial^{2}\mathcal{R}ic}{\partial y^{i}\partial y^{j}}=R^{\alpha\beta}\dot{\nabla}_{\alpha}\partial_{\beta}\mathcal{R}ic
+2F^{-2}A^{k}_{ij}R^{ij}y^{\lambda}_{k}\partial_{\lambda}\mathcal{R}ic.
\end{equation*}
Putting $H^{\lambda}:=-2A^{k}_{ij}R^{ij}y^{\lambda}_{k}$, we can rewrite (\ref{eq.6}) on $SM$ as follows
\begin{equation}\label{eq.9+4}
\frac{\partial}{\partial t}\mathcal{R}ic=-F^{2}R^{\alpha\beta}\dot{\nabla}_{\alpha}\partial_{\beta}\mathcal{R}ic
+H^{\lambda}\partial_{\lambda}\mathcal
{R}ic.
\end{equation}
By means of (\ref{eq.9+4}) one can write the following inequality
\begin{equation}\label{eq.10}
\frac{\partial}{\partial t}\mathcal{R}ic\geq-F^{2}R^{\alpha\beta}\dot{\nabla}_{\alpha}\partial_{\beta}\mathcal{R}ic
+H^{\lambda}\partial_{\lambda}\mathcal
{R}ic-\mathcal{R}ic^{2}.
\end{equation}
By assumption $(M,F(0))$ has positive flag curvature. Definition of the flag curvature (\ref{flag}) shows that $R_{g_{0}}>0$. Hence, Proposition \ref{asli2} implies that $R^{\alpha\beta}(t)$ is positive definite for all $t\in[0,T)$. Therefore, inequality (\ref{eq.10}) is an inequality of parabolic type. Let $\phi$ be a solution to the ODE
\begin{equation}\label{ode}
\frac{d}{dt}\phi=-\phi^{2},
\end{equation}
with initial value $\phi(0)=\inf_{SM}\mathcal{R}ic_{g(0)}=\alpha$. Equation (\ref{ode}) is a Bernoulli equation and its exact solution is
\begin{equation*}
\phi(t)=\frac{\alpha}{1+\alpha t}.
\end{equation*}
Using the weak minimum principle, in the sense of Corollary \ref{minimum principle}, and the inequality (\ref{eq.10}) we conclude that $\mathcal{R}ic_{g(t)}\geq \frac{\alpha}{1+\alpha t}.$
 This completes the proof of Theorem \ref{asli5}.\hspace{\stretch{1}}$\Box$

 
{\small \it  Behroz Bidabad, bidabad@aut.ac.ir}\\
{\small \it Maral Khadem Sedaghat,
m\_sedaghat@aut.ac.ir \\ Faculty of Mathematics and Computer Science, Amirkabir University of Technology (Tehran Polytechnic), Hafez Ave., 15914 Tehran, Iran.}

\end{document}